\documentclass[11pt]{amsart}

\usepackage[dvips]{graphicx}
\usepackage{calc}
\usepackage{color}
\usepackage{amsmath}
\usepackage{amssymb}
\usepackage{amscd}
\usepackage{amsthm}
\usepackage{amsbsy}
\usepackage{delarray}
\usepackage{enumerate}
\usepackage[T1]{fontenc}
\usepackage{inputenc}
\usepackage{enumerate}
\usepackage{hyperref}
\usepackage[margin= 1.5 in]{geometry}

\begin{document}

\setlength{\parindent}{5mm}
\renewcommand{\leq}{\leqslant}
\renewcommand{\geq}{\geqslant}
\newcommand{\N}{\mathbb{N}}

\newcommand{\Z}{\mathbb{Z}}
\newcommand{\R}{\mathbb{R}}
\newcommand{\C}{\mathbb{C}}
\newcommand{\F}{\mathbb{F}}
\newcommand{\g}{\mathfrak{g}}
\newcommand{\h}{\mathfrak{h}}
\newcommand{\K}{\mathbb{K}}
\newcommand{\RN}{\mathbb{R}^{2n}}
\newcommand{\ci}{c^{\infty}}
\newcommand{\derive}[2]{\frac{\partial{#1}}{\partial{#2}}}
\renewcommand{\S}{\mathbb{S}}
\renewcommand{\H}{\mathbb{H}}
\newcommand{\eps}{\varepsilon}
\renewcommand{\r} {\textbf{r}}
\renewcommand{\L}{\mathcal{L}}

\newcommand{\B}{B^{2n}(\frac{1}{\sqrt{\pi}})}

\theoremstyle{plain}
\newtheorem{theo}{Theorem}
\newtheorem{prop}[theo]{Proposition}
\newtheorem{lemma}[theo]{Lemma}
\newtheorem{definition}[theo]{Definition}
\newtheorem*{notation*}{Notation}
\newtheorem*{notations*}{Notations}
\newtheorem{corol}[theo]{Corollary}
\newtheorem{conj}[theo]{Conjecture}
\newtheorem{question}[theo]{Question}
\newtheorem*{question*}{Question}
\newtheorem*{remark*}{Remark}

\newenvironment{demo}[1][]{\addvspace{8mm} \emph{Proof #1.
    ---~~}}{~~~$\Box$\bigskip}

\newlength{\espaceavantspecialthm}
\newlength{\espaceapresspecialthm}
\setlength{\espaceavantspecialthm}{\topsep} \setlength{\espaceapresspecialthm}{\topsep}

\newenvironment{example}[1][]{\refstepcounter{theo} 
\vskip \espaceavantspecialthm \noindent \textsc{Example~\thetheo
#1.} }%
{\vskip \espaceapresspecialthm}

\newenvironment{remark}[1][]{\refstepcounter{theo} 
\vskip \espaceavantspecialthm \noindent \textsc{Remark~\thetheo
#1.} }%
{\vskip \espaceapresspecialthm}

\def\Homeo{\mathrm{Homeo}}
\def\Hameo{\mathrm{Hameo}}
\def\Diffeo{\mathrm{Diffeo}}
\def\Symp{\mathrm{Symp}}
\def\Id{\mathrm{Id}}
\newcommand{\norm}[1]{||#1||}
\def\Ham{\mathrm{Ham}}
\def\Hamtilde{\widetilde{\mathrm{Ham}}}
\def\Crit{\mathrm{Crit}}
\def\Spec{\mathrm{Spec}}
\def\osc{\mathrm{osc}}
\def\Cal{\mathrm{Cal}}

\title[Hofer's distance on Lagrangians]{Unboundedness of the Lagrangian Hofer distance in the Euclidean ball} 
\author{Sobhan Seyfaddini}
\date{\today}

\address{SS: D\'epartement de Math\'ematiques et Applications de l'\'Ecole Normale Sup\'erieure, 45 rue d'Ulm, F 75230 Paris cedex 05}
\email{sobhan.seyfaddini@ens.fr}


\begin{abstract} 
Let $\mathcal{L}$ denote the space of Lagrangians Hamiltonian isotopic to the standard Lagrangian in the unit ball in $\mathbb{R}^{2n}$.  We prove that the Lagrangian Hofer distance on $\mathcal{L}$ is unbounded.
\end{abstract}

\maketitle

\section{Introduction}
  Let $(M, \omega)$ be a symplectic manifold and denote by $C^{\infty}_c([0,1]\times M)$ the set of compactly supported Hamiltonians on $M$.  Each Hamiltonian $H \in C^{\infty}_c([0,1]\times M)$ generates a Hamiltonian flow $\phi^t_H$. The group of (compactly supported) Hamiltonian diffeomorphisms of $M$, denoted  by $Ham(M)$, is defined to be the set of time-1 maps of such flows.  The Hofer norm of $\psi \in Ham(M)$ is defined by the expression $\Vert \psi \Vert = \inf \{ \Vert H \Vert_{(1,\infty)}: \psi = \phi^1_H \},$
	where $\displaystyle \Vert H \Vert_{(1, \infty)} = \int_0^1{ (\max_M H(t, \cdot) - \min_M H(t, \cdot)) \,dt}.$
	
	Let $B^{2n} \subset \mathbb{R}^{2n}$ denote the open unit ball equipped with the symplectic structure $\omega_0 = \frac{1}{\pi} \sum_{i=1}^{n} dx_i \wedge dy_i$.  Denote by $L_0 = \{(x_i, y_i) \in B : y_i = 0 \} \cap B^{2n}$ the standard Lagrangian in the unit ball, and let $\L = \{ \phi(L_0): \phi \in Ham(B^{2n}) \}$.  We endow $\L$ with the Lagrangian Hofer distance $d$, defined by: $$\forall \;\; L_1, L_2 \in \L \;\;\; d(L_1, L_2) = \inf \{\Vert \phi \Vert : \phi \in Ham(B^{2n}) \text{, }\phi(L_1) = L_2 \}.$$ 
	
	In \cite{Kh09}, Khanevsky proved that, in dimension $2$, the metric space $(\L, d)$ is unbounded.  In higher dimensions, this seemingly basic question has remained open despite the fact that much progress has been made in the field of Lagrangian Hofer geometry; see for example \cite{hum12, Us2, Zap12}.  Our main goal here is to prove the unboundedness of $(\L, d)$ in full generality.  Let $C^{\infty}_c([0,1])$ denote the space of smooth and compactly supported functions on $[0,1]$.  For $f \in C^{\infty}_c([0,1])$ define $\Vert f \Vert_{\infty} = \max_{[0,1]} |f|$. Here is our main result. 
	
	\begin{theo} \label{theo_main}
	There exists a map $\Psi: C^{\infty}_c([0,1]) \rightarrow \L$ and positive constants $C, D \in \mathbb{R}$ such that 
	$$\frac{\Vert f-g \Vert_{\infty} - C}{D} \leq d(\Psi(f), \Psi(g)) \leq  \Vert f-g \Vert_{\infty}.$$
	In particular, the metric space $(\L, d)$ is unbounded.
	\end{theo}
	
	\begin{remark*}
	I have recently been informed that Michael Khanevsky and Frol Zapolsky have jointly obtained a different proof of the unboundedness of $(\L, d)$.  Their proof relies on a forthcoming work of R\'emi Leclercq and Frol Zapolsky on Lagrangian spectral invariants \cite{LZ}.
	\end{remark*}
	
	\section{Proof of Theorem \ref{theo_main}}
	This paper, like Khanevsky's, relies heavily on Entov and Polterovich's work on the theory of quasimorphisms \cite{EP03, BEP, EP06, EP09, EPP}.  A quasimorphism on a group $G$ is a function $\mu : G \rightarrow \mathbb{R}$ such that $|\mu(ab) - \mu(a) - \mu(b)| \leq D,\;\; \forall a,b \in G$, where $D$ is a constant which is usually referred to as the defect of the quasimorphism $\mu$.  If $\mu(a^k) = k \mu(a)$ for all $a \in G$ and $k \in \mathbb{Z}$, then $\mu$ is called a homogeneous quasimorphism.
	
	Let $\B$ denote the open Euclidean ball of radius $\frac{1}{\sqrt{\pi}}$.  For $0 < r < \frac{1}{n \pi}$, let  $T(r) = \{(z_1, \cdots, z_n) \subset \C^n: |z_1|^2 = \cdots = |z_n|^2 = r\}$.  Note that $T(r)$ is a Lagrangian torus in $\B$.  In Section \ref{construction}, we will construct a family of homogenous quasimorphisms $\eta_{\delta}: Ham(\B) \rightarrow \mathbb{R}$; the parameter $\delta$ will range over the interval  $(\frac{n}{n+1},1]$.  Note that for values of $\delta > \frac{n}{n+1}$, the Lagrangian torus  $T(\frac{1}{\delta \pi (n+1)})$ is contained in $\B$.
	
\begin{theo}\label{theo_quasi}
	The homogeneous quasimorphisms $\eta_{\delta}: Ham(\B) \rightarrow \mathbb{R}$ have the following properties	
	\begin{enumerate}
	\item \label{Hofer_Lip} There exists a constant $C > 0$, independent of $\delta$, such that  $|\eta_{\delta}(\phi)| \leq C \Vert \phi \Vert$.
	
	\item \label{super_heavy} If $F: [0,1] \times \B \rightarrow \mathbb{R}$ is a compactly supported Hamiltonian such that $F(t,x)\leq c$ (res. $F(t,x)\geq c$) for all $(t,x) \in [0,1] \times T(\frac{1}{\delta \pi (n+1)})$, then $\eta_{\delta}(\phi^1_F) \leq c$ (res. $\eta_{\delta}(\phi^1_F) \geq c$). 
	\item \label{descent} There exists a constant $D >0$, independent of $\delta$, such that if $\phi(L_0) = \psi(L_0)$, then $|\eta_{\delta}(\phi) - \eta_{\delta}(\psi)| \leq D$.
	\end{enumerate}
	\end{theo}
	
	We will next use the above theorem to prove Theorem \ref{theo_main}.  The fact that Theorem \ref{theo_main} follows from the existence of quasimorphisms with the above list of properties is implicitly present in Khanevsky's paper \cite{Kh09}.  Various versions of the quasimorphisms employed in this note have appeared in the work of Entov, Polterovich and their collaborators; see for example \cite{EP03, BEP, EP06, EP09, EPP}.  This close connection to Entov and Polterovich's work allows us to prove the first two properties by standard techniques.  The non-standard part is establishing Property \eqref{descent}, which in a sense states that, up to a bounded error, $\eta_{\delta}(\phi)$ is an invariant of the Lagrangian $\phi(L_0)$.  In \cite{Kh09}, Khanevsky proves this property using two dimensional methods.  We will use a different approach which is applicable in all dimensions.
	
	\begin{proof}[Proof of Theorem \ref{theo_main}]
	We will prove the theorem for $\B$ rather than $B^{2n}$.  Of course, this is sufficient as $\B$ and $B^{2n}$ are conformally symplectomorphic.  We will continue to denote by $L_0$ the restriction of $\{(x_i, y_i) \in B : y_i = 0 \}$ to $\B$ and by $\mathcal{L}$ the space of Lagrangians which are Hamiltonian isotopic to $L_0$ inside $\B$.
	
	Take $\phi \in Ham(\B)$ and let $\psi$ denote any other Hamiltonian diffeomorphism such that $\phi(L_0) = \psi(L_0)$.  Using Properties (\ref{Hofer_Lip}) and (\ref{descent}) we obtain the following for all $\delta \in (\frac{n}{n+1},1]$:
	$$|\eta_\delta(\phi)| - D \leq |\eta_\delta(\psi)| \leq C \Vert \psi \Vert.$$
	The above implies that $\forall \phi \in Ham(\B)$ and $\forall \delta \in (\frac{n}{n+1},1]$ we have
	\begin{equation}\label{eq:lowerbound}
	\frac{|\eta_\delta(\phi)| - D}{C} \leq d(L_0, \phi(L_0)).
	\end{equation}
	
	Let $J = [\frac{n}{n+1}, 1 ]$ and denote by $C^{\infty}_c(J)$ the set of functions whose support is compactly contained in the interior of $J$.   We will construct a map $\Psi: C^{\infty}_c(J) \rightarrow \L$ such that $\frac{\Vert f-g \Vert_{\infty} - C}{D} \leq d(\Psi(f), \Psi(g)) \leq  \Vert f-g \Vert_{\infty}.$  Of course, this establishes Theorem \ref{theo_main} as $J$ is diffeomorphic to $[0,1]$.
	
	For any $f \in C^{\infty}_c(J)$, set $\tilde{f}(z) = f(\pi |z|^2),$ where $z \in \C^n$. Note that the Hamiltonian $\tilde{f}$ is compactly supported in $\B.$ We define $\Psi: C^{\infty}_c(J) \rightarrow \L$ by the following expression: $$\Psi(f) = \phi^1_{\tilde{f}}(L_0).$$

	 We will first show that $\frac{ \Vert f -g \Vert_{\infty} - D}{C} \leq d(\Psi(f), \Psi(g))$.  Apply Inequality \eqref{eq:lowerbound} to $\phi^{-1}_{\tilde{g}} \phi^1_{\tilde{f}}$, where $\phi^{-1}_{\tilde{g}}$ denotes the inverse of $\phi^{1}_{\tilde{g}}$, to obtain 
	
	$$\frac{|\eta_\delta(\phi^{-1}_{\tilde{g}} \phi^1_{\tilde{f}})| - D}{C} \leq d(L_0, \phi^{-1}_{\tilde{g}} \phi^1_{\tilde{f}}(L_0)), \;\; \forall \delta \in (\frac{n}{n+1}, 1].$$
	
Observe that $\tilde{f}$ and $\tilde{g} $ Poisson commute and hence $\phi^{-1}_{\tilde{g}} \phi^1_{\tilde{f}} = \phi^1_{\tilde{f}-\tilde{g}}.$  Also, by symplectic invariance of Hofer's distance $d(L_0, \phi^{-1}_{\tilde{g}} \phi^1_{\tilde{f}}(L_0)) = d(\phi^1_{\tilde{g}}(L_0), \phi^1_{\tilde{f}}(L_0))$.  Therefore, the previous inequality is equivalent to 
	$$\frac{ |\eta_\delta(\phi^1_{\tilde{f}-\tilde{g}})| - D}{C} \leq  d(\phi^1_{\tilde{g}}(L_0), \phi^1_{\tilde{f}}(L_0)), \;\; \forall \delta \in (\frac{n}{n+1}, 1].$$

Note that $\tilde{f}$ and $\tilde{g}$ are constant on each of the tori $T(\frac{1}{\delta \pi (n+1)})$.  Hence, using Property \eqref{super_heavy} from Theorem \ref{theo_quasi}, we conclude that $\eta_\delta(\phi^1_{\tilde{f}-\tilde{g}}) = f(\frac{n}{\delta(n+1)})-g (\frac{n}{\delta(n+1)})$.  Picking $\delta$ so that $f-g$ attains its maximum at $\frac{n}{\delta(n+1)}$ yields :
	
$$\frac{ \Vert f -g \Vert_{\infty} - D}{C} \leq  d(\phi^1_{\tilde{g}}(L_0), \phi^1_{\tilde{f}}(L_0)).$$

It remains to prove that $d(\Psi(f), \Psi(g)) \leq \Vert f -g \Vert_{\infty},$ for any $f, g \in C^{\infty}_c(J)$. Indeed, 
$$d(\Psi(f), \Psi(g)) = d(\phi^1_{\tilde{f}}(L_0), \phi^1_{\tilde{g}}(L_0)) =  d(L_0, (\phi^1_{\tilde{f}})^{-1} \phi^1_{\tilde{g}}(L_0))\leq || (\phi^1_{\tilde{f}})^{-1} \phi^1_{\tilde{g}}||$$ $$ \leq \Vert \tilde{f} - \tilde{g} \Vert_{\infty} = \Vert f - g \Vert_{\infty}.$$ 
This completes our proof.
	\end{proof}
	
	\section{Construction of the Quasimorphisms}\label{construction}  
	In this section, following Biran, Entov, and Polterovich \cite{BEP}, we construct the quasimorphisms $\eta_\delta$.
	
	\subsection{Spectral invariants on $\mathbb{C}P^n$} \label{subsec: spectral_invariants}
	Let $\omega_{FS}$ denote the Fubini-Study symplectic form on $\mathbb{C}P^n$ normalized so that $Vol(\C P^n)=1$.  Throughout the rest of this article we assume that $\mathbb{C}P^n$ is equipped with the symplectic structure induced by $\omega_{FS}$.
	
  It follows from the works of  C. Viterbo, M. Schwarz, and Y.-G. Oh \cite{viterbo1, schwarz, oh2} that one can associate to each Hamiltonian $H \in C^{\infty}([0,1] \times \mathbb{C}P^n)$ and each non-zero quantum homology class $ a \in QH_*(\mathbb{C}P^n)) \setminus \{0\}$  a so called spectral invariant $c(a, H) \in \mathbb{R}.$  These invariants are constructed via the machinery of Hamiltonian Floer theory and, in fact, they can be defined on a very large class of symplectic manifolds, however, we will only be concerned with $\mathbb{C}P^n$.  The only quantum homology class that we will be dealing with is the fundamental class $[\mathbb{C}P^n]$; for brevity we will denote $c(H) = c([\C P^n], H)$. 
	
	For our purposes we must introduce the asymptotic version of the spectral invariant $c$.  This is defined as follows: for $H, G \in C^{\infty}([0,1] \times \C P^n)$ define $H\#G(t,x) = H(t,x) + G(t, (\phi^t_H)^{-1}(x))$; the flow of $H\#G$ is $\phi^t_H \circ \phi^t_G$.  Following Entov and Polterovich \cite{EP03}  we define the asymptotic spectral invariant of a Hamiltonian $H$ by 
  $$\zeta(H) = \lim_{k \to \infty} \frac{c(H^{\#k})}{k}.$$
	Our asymptotic spectral invariant, $\zeta$, is defined for all time-dependent Hamiltonians.  The restriction of $\zeta$ to the set of time-independent Hamiltonians, i.e. $C^{\infty}(\C P^n)$, is what Entov and Polterovich refer to as a symplectic quasi-state on $\C P^n$.  In fact, the restriction of $\zeta$ to $C^{\infty}(\C P^n)$ is precisely the quasi-state constructed on $\C P^n$ in \cite{EP06}. 
	
	According to Entov and Polterovich \cite{EP09}, a closed subset $X \subset \C P^n$ is said to be superheavy with respect to $\zeta$ if 
	$$\inf_X H \leq \zeta(H) \leq  \sup_X H \;\; \forall H \in C^{\infty}(\C P^n).$$ The Clifford torus and $\R P^n$ are two examples of superheavy subsets of $\C P^n$.  Superheavyness of these two sets follows from Theorems 1.6 and 1.15 of \cite{EP09}, respectively.   In \cite{EP09}, the notion of superheavyness is defined only for autonomous Hamiltonians.  Using standard properties of spectral invariants one can easily show that if $X \subset \C P^n$ is superheavy, then
	
	\begin{equation}\label{eq:superheavy} \inf_X H \leq \zeta(H) \leq  \sup_X H \;\; \forall H \in C^{\infty}([0,1]  \times \C P^n) .\end{equation}
	
\subsection{Entov and Polterovich's Calabi quasi-morphism}\label{subsec: cal_quasi}
	In \cite{EP03}, Entov and Polterovich construct a quasimorphism $\mu: Ham(\C P^n) \rightarrow \R$ which is defined by the following expression:  
	\begin{equation}\label{def_mu}
	\mu(\phi^1_F) =  - \zeta(F) + \int_0^1 \int_{\C P^n} F(t,x) \, \omega_{FS}^n \; dt.
	\end{equation}

In \cite{EP03}, it is proven that $\mu$ does not depend on the choice of the generating Hamiltonian $F$ and hence it is a well-defined map from $Ham(\C P^n)$ to $\R$. Furthermore, $\mu$ is a homogeneous quasimorphism.  The relationship between this quasimorphism and the aforementioned quasi-state $\zeta$ is discussed in detail in \cite{EP06}.  For example, the above formula relating $\zeta$ and $\mu$ can be found in Section 6 of \cite{EP06}.

Recall that a subset $U$ of a symplectic manifold is said to be displaceable if there exists a Hamiltonian diffeomorphism $\psi$ such that $U \cap \psi(U) = \emptyset$.  If the support of a Hamiltonian $F$ is displaceable, then\begin{equation}\label{eq:vanishing} \zeta(F) = 0, \end{equation} and hence, $\mu(\phi^1_F) =  \int_0^1 \int_{\C P^n} F(t,x) \, \omega_{FS}^n \; dt.$  This is referred to as the Calabi property of $\mu$ and for this reason $\mu$ is often referred to as a Calabi quasimorphism; for further details see \cite{EP03, EP06}.

\subsection{Constructing $\eta_\delta$}\label{subsec: construct}
In this section we closely follow Section 4 of \cite{BEP}.  For $0< \delta \leq 1$, define embeddings $\theta_\delta : \B \rightarrow \C P^n$ by the formula:
$$\displaystyle (z_1, \cdots, z_n) \mapsto [\sqrt{\frac{1}{\pi} - \Sigma_{i=1}^{n} \delta |z_i|^2}: \sqrt{\delta} z_1: \cdots : \sqrt{\delta} z_n],$$
where $z_i's$ denote the standard complex coordinates on $\C^n$.  The maps $\theta_\delta$ pull $\omega_{FS}$ back to $\delta \, \omega_0$ and so these embeddings are all conformally symplectic.  Clearly, $\theta_1$ is a genuine symplectic embedding.  

As displayed in \cite{BEP}, for $\delta > \frac{n}{n+1}$ the torus $T(\frac{1}{\delta \pi (n+1)}) \subset \B$ is mapped by $\theta_\delta$ onto the Clifford torus.  It is also clear that $\theta_\delta (L_0)$ is contained inside $\R P^n \subset \C P^n$.

  For $\delta > \frac{n}{n+1}$, we define $\eta_\delta : Ham(\B) \rightarrow \R$ as follows: Given $\phi \in Ham(\B)$ pick a Hamiltonian $F \in C^{\infty}_c([0,1] \times \B)$ such that $\phi^1_F = \phi$.  Define \begin{equation}\label{eta_def2}
\eta_\delta(\phi) = \delta^{-1} \zeta( \delta F \circ \theta_\delta^{-1}).
\end{equation}
We must show that $\eta_\delta$ does not depend on the choice of the generating Hamiltonian $F$. One can easily check that the time-1 map of the Hamiltonian $\delta F \circ \theta_\delta^{-1}: [0,1] \times \C P^n \rightarrow \R$ is the Hamiltonian diffeomorphism $\theta_\delta \phi \theta_\delta^{-1}$.   This combined with Equation \eqref{def_mu} yields: $$\delta^{-1} \zeta( \delta F \circ \theta_\delta^{-1}) = - \delta^{-1} \mu(\theta_\delta \phi \theta_\delta^{-1}) + \delta^{-1} \int_0^1 \int_{\C P^n} \delta F(t,\theta_\delta^{-1}(x)) \; \omega_{FS}^n.$$ A simple computation shows that $$\int_0^1 \int_{\C P^n} \; F(t,\theta_\delta^{-1}(x)) \; \omega_{FS}^n = \delta^{n} \int_0^1 \int_{\B} F(t,x) \; \omega_0^n = \delta^{n} Cal(\phi),$$
where $Cal: Ham(\B) \rightarrow \R$ denotes the Calabi homomorphism: $$Cal(\phi^1_F) =  \int_0^1 \int_{\B} F(t,x) \; \omega_0^n \; dt.$$
 Hence, we have obtained the following alternative definition for $\eta_\delta$:
\begin{equation}\label{eta_def1}
\eta_{\delta}(\phi) = - \delta^{-1} \mu(\theta_\delta \phi \theta_\delta^{-1}) + \delta^{n} Cal(\phi).
\end{equation}
It is clear from the above formula that $\eta_\delta$ is well-defined and does not depend on the choice of the generating Hamiltonian $F$.  Furthermore, $\eta_\delta$ is a quasimorphism as it is a linear combination of quasimorphisms.  Results from \cite{Sey12} on descent of asymptotic spectral invariants provide an alternative method for proving that $\eta_\delta$ is a well-defined quasimorphism.  Additionally, one can use Formula \eqref{eta_def1} to show that the quasimorphisms $\eta_\delta$ have bounded defect.  Indeed, Formula \eqref{eta_def1}, combined with the fact the $Cal$ is a homomorphism, implies that
$$\eta_\delta(\phi \psi) - \eta_\delta(\phi) - \eta_\delta(\psi) = - \delta^{-1} (\mu(\theta_\delta \phi \psi \theta_\delta^{-1}) - \mu(\theta_\delta \phi \theta_\delta^{-1}) - \mu(\theta_\delta \psi \theta_\delta^{-1})).$$ From this we obtain that
\begin{equation} \label{bdd_defect}
|\eta_\delta(\phi \psi) - \eta_\delta(\phi) - \eta_\delta(\psi)| \leq  \delta^{-1} Def(\mu) \leq \frac{n+1}{n} Def(\mu), \end{equation}
where $Def(\mu)$ denotes the defect of $\mu$.  The last inequality holds because $\delta \in (\frac{n}{n+1}, 1]$.

\subsection{Proof of Theorem \ref{theo_quasi}}\label{subsec: proof}
 Part {\eqref{Hofer_Lip} of the theorem follows from Formula \eqref{eta_def1} and the fact that both $\mu$ and $Cal$ are Lipschitz continuous with respect to Hofer's norm.  Part \eqref{super_heavy} follows from Formula \eqref{eta_def2}:  Indeed, if $F \leq c$ on $T(\frac{n}{\pi \delta (n+1)})$, then $\delta F \circ \theta_\delta^{-1} \leq \delta c$ on the Clifford torus.  Since the Clifford torus is superheavy for $\zeta$ we see that $\delta ^{-1} \zeta(\delta F \circ \theta_\delta^{-1}) \leq c.$  

It remains to prove Part \eqref{descent} of Theorem \ref{theo_main}. Let $D = \frac{n+1}{n} Def(\mu)$.  Inequality \eqref{bdd_defect} implies that $|\eta_\delta(\phi^{-1} \psi) + \eta_\delta(\phi) - \eta_\delta(\psi)| \leq D.$  The Hamiltonian diffeomorphism $\phi^{-1} \psi$ preserves $L_0$ and hence, it is sufficient to prove that $\eta_\delta$ vanishes on the set of Hamiltonian diffeomorphisms which preserve $L_0$.

Suppose that $\phi^1_F(L_0) = L_0.$   Pick a positive number $\epsilon$ such that $B^{2n}(\epsilon)$, the ball of radius $\epsilon$, is displaceable inside $\B$.  For each $s \in [\epsilon, 1]$ we define a Hamiltonian diffeomorphism $\psi_s$ as follows:
    $$\psi_s(x) =  \left\{ \begin{array}{ll}
         s\phi^1_F(\frac{x}{s}) & \mbox{if $|x| \leq s $};\\
         x & \mbox{if $|x| \geq s$}.\end{array} \right. $$
    A simple computation shows that $\psi_s \in Ham(\B)$, and in fact, $\psi_s$ is the time-1 map of the flow of the following Hamiltonian:
    $$F_s(t,x) =  \left\{ \begin{array}{ll}
         s^2F(t,\frac{x}{s})& \mbox{if $|x| \leq s $};\\
         0 & \mbox{if $|x| \geq s$}.\end{array} \right. $$
Furthermore, the Hamiltonian diffeomorphisms $\psi_s$ all preserve $L_0$ because $\phi^1_F(L_0) = L_0$.  Now, consider the path of Hamiltonian diffeomorphisms $[0,1] \rightarrow Ham(\B)$ defined by
$$t \mapsto \psi_{\epsilon}^{-1} \psi_{t (1-\epsilon) + \epsilon}.$$
Pick a Hamiltonian $H\in C^{\infty}_c([0,1] \times \B)$ such that $$\phi^t_H = \psi_{\epsilon}^{-1} \psi_{t(1-\epsilon) + \epsilon}.$$
Because $\phi^t_H(L_0) = L_0$ for each $t \in [0,1]$, the Hamiltonian vector fields $\{X_{H_t}\}_{t\in[0,1]}$ must all be tangential to $L_0$.  Since $L_0$ is a Lagrangian, we conclude that $dH_t = \omega_0(X_{H_t}, \cdot)$ vanishes on the tangent space to $L_0$.  This implies that for each $t\in [0,1]$ the restriction of $H_t$ to $L_0$ is constant.  Now, $H_t$ is compactly supported in the interior of $\B$ and thus we conclude $H_t|_{L_0} = 0 \; \forall t \in [0,1].$ As a consequence the Hamiltonian $\delta H \circ \theta_\delta^{-1}$ vanishes on $\R P^n \subset \C P^n$; this is because $\theta_\delta(L_0) \subset \R P^n$.  Since $\R P^n$ is superheavy, we conclude, by Inequality \eqref{eq:superheavy} that $$\eta_\delta(\phi^1_H) = 0.$$

Next, note that $\psi_{\epsilon}$ is supported in $B^{2n}_{\epsilon}$ which is displaceable inside $\B$.  Therefore, Equation \eqref{eq:vanishing} yields that $$\eta_\delta(\psi_\epsilon) = 0.$$
Since $\phi^1_F = \psi_\epsilon \phi^1_H$ and $\eta_\delta$ is a quasimorphism, it follows that
 $$|\eta_\delta(\phi^1_F)| = |\eta_\delta(\phi^1_F)  - \eta_\delta(\psi_\epsilon) -\eta_\delta(\phi^1_H)| \leq D.$$
Thus far, we have proven that if $\phi^1_F(L_0) = L_0$, then $|\eta_\delta(\phi^1_F)| \leq D$.  Of course, if $\phi^1_F(L_0) = L_0,$ then $(\phi^1_F)^k(L_0) = L_0$ for any $k \in \mathbb{Z}$ and therefore $|\eta_\delta((\phi^1_F)^k)| \leq  D$.  But $\eta_\delta$ is homogeneous and so $\eta_\delta((\phi^1_F)^k) = k \eta_\delta(\phi^1_F).$ It follows that $|k \eta_\delta(\phi^1_F)| \leq D$ for all $k \in \mathbb{Z}$ and so $\eta_\delta(\phi^1_F) = 0.$
 This completes our proof.

\subsection*{Aknowledgements} In writing this paper, I benefited greatly from many stimulating conversations with Michael Khanevsky.  I am grateful to him for his time and valuable input.  Many of these conversation took place at the Institute for Advanced Study during the 2011-12 program on Symplectic Dynamics.  I would like to thank Helmut Hofer for inviting me to the institute, members and employees of the institute for their warm hospitality, and Alan Weinstein for his support which made my stay at the institute possible.  Lastly, for helpful comments and conversations, I am thankful to Lev Buhovsky, Vincent Humili\`ere, Leonid Polterovich, and Michael Usher.

\bibliographystyle{abbrv}
\bibliography{biblio}

\end{document}